\documentclass[reqno]{amsart}
\usepackage{cite}
\usepackage[russian,english]{babel}
\usepackage{fixltx2e}
\usepackage{amssymb,amsmath,amsfonts,amsthm}
\usepackage[utf8]{inputenc}
\usepackage{cmap}
\usepackage[T1]{fontenc}
\usepackage[all]{xy}
\usepackage{mathbbol}
\usepackage{graphicx}
\usepackage{xcolor}
\usepackage{lineno}

\newtheorem{theorem}[subsection]{Theorem}
\newtheorem{proposition}[subsection]{Proposition}

\newtheorem{corollary}[subsection]{Corollary}
\theoremstyle{definition}
\newtheorem{definition}[subsection]{Definition}
\newtheorem{example}[subsection]{Example}

\theoremstyle{remark}
\newtheorem{remark}[subsection]{Remark}

\begin{document}

\title[Almost inner derivations of some nilpotent Leibniz Algebras]{Almost inner derivations of some nilpotent Leibniz Algebras}

\author{J.K.Adashev}
\address{Institute of Mathematics of Uzbek Academy of Sciences, Mirzo Ulugbek 81, 100170, Tashkent, Uzbekistan.}
\email{adashevjq@mail.ru}

\author{T.K.Kurbanbaev}
\address{Karakalpak state university, Ch.Abdirov 1, 230100, Nukus, Uzbekistan.}
\email{tuuelbay@mail.ru}

\begin{abstract}
We investigate almost inner derivations of some finite-dimensional nilpotent Leibniz algebras. We show the existence of almost inner derivations of Leibniz filiform non-Lie algebras differing from inner derivations, we also show that the almost inner derivations of some filiform Leibniz algebras containing filiform Lie algebras do not coincide with inner derivations.

\end{abstract}

\subjclass[2010]{17A32, 17A60, 17B10}
\keywords{Leibniz algebra, derivation, inner derivation, almost derivation}

\maketitle

\section{Introduction}

Lie algebra is an algebra satisfying the anticommutativity identity and the Jacobi identity. The derivations of finite-dimensional Lie algebras are a well-studied direction of the theory of Lie algebras. It should be noted that the space of all derivations of Lie algebras is also Lie algebra with respect to the commutator. In the set of derivations of Lie algebras, there exist subsets of the so-called inner derivations. Naturally, there is a question: in what classes of algebras do derivations exist? and which are not inner? For the semisimple Lie algebras the sets of inner derivations and derivations coincide \cite{Hump}.

Almost inner derivations of Lie algebras were introduced by C.S. Gordon and E.N. Wilson \cite{Gordon} in the study of isospectral deformations of compact manifolds. Gordon and Wilson wanted to construct not only finite families of isospectral nonisometric manifolds, but rather continuous families. They constructed isospectral but nonisometric compact Riemannian manifolds of the form $G/\Gamma$, with a simply connected exponential solvable Lie group $G$, and a discrete cocompact subgroup $\Gamma$ of $G$. For this construction, almost inner automorphisms
and almost inner derivations were crucial.

Gordon and Wilson considered not only almost-inner derivations, but they studied almost inner automorphisms of Lie groups. The concepts of "almost inner" automorphisms and derivations, almost homomorphisms or almost conjugate subgroups arise in many contexts in algebra, number theory and geometry. There are several other studies of related concepts, for example, local derivations, which are a generalization of almost inner derivations and automorphisms \cite{Ayupov0}, \cite{Ayupov01}.

In \cite{AK15} we initiated the study of derivation type maps on non-associative algebras, namely, we investigated so-called  2-local derivations on finite-dimensional Lie algebras, and showed an essential difference between semisimple and nilpotent Lie algebras is the behavior of their 2-local derivations. The present paper is devoted to local derivation on finite-dimensional Lie algebra over an algebraically closed field of characteristic zero.

Local derivation first was considered in 1990, Kadison \cite{Kad90} and Larson and Sourour \cite{LarSou}. Let $X$ be a Banach $A$-bimodule over a Banach algebra $A$, a linear mapping $\Delta:A\to X$ is said to be a \textit{local derivation} if for every $x$ in $A$ there exists a derivation $D_x :A\to X$, depending on $x$, satisfying $\Delta(x) = D_x (x).$

The main problems concerning this notion are to find conditions under which local derivations become derivations and to present examples of algebras with local derivations that are not derivations \cite{Bresar, Kad90, LarSou}. Kadison proves in \cite[Theorem A]{Kad90} that each continuous local derivation of a von Neumann algebra $M$ into a dual Banach $M$-bimodule is a derivation. This theorem gave rise to studies and several results on local derivations on C$^*$-algebras, culminating with a definitive contribution due to Johnson, which asserts that every continuous local derivation of a C$^*$-algebra $A$ into a Banach $A$-bimodule is a derivation \cite[Theorem 5.3]{John01}. Moreover in his paper, Johnson also gives an automatic continuity result by proving that local derivations of a C$^*$-algebra $A$ into a Banach $A$-bimodule $X$ are continuous even if not assumed a priori to be so (cf. \cite[Theorem 7.5]{John01}).

In the theory of Lie algebras, there is a theorem which says that in the finite-dimensional nilpotent Lie algebra there
are not inner (i.e. outer) derivations \cite{Jacobson2}. We give an example \ref{AidOth} to shows that that there exists 4-dimensional nilpotent Lie algebras, where any almost inner derivation is an outer derivation, and the converse is true also.
But this question is still open for the general case. In \cite{Burde} authors study almost inner derivations of some nilpotent Lie algebras. Prove the basic properties of almost inner derivations, calculate all almost inner derivations of Lie algebras for small dimensions. They also introduced the concept of fixed basis vectors for nilpotent Lie algebras defined by graphs and studied free nilpotent Lie algebras of the nilindex 2 and 3.

We recall that the study of almost-inner derivations of the Leibniz algebras is an open problem. Therefore in this paper we consider almost-inner derivations for some nilpotent Leibniz algebras. We prove the basic properties of almost inner derivations of the Leibniz algebras. We get almost all  inner derivations of four-dimensional nilpotent Leibniz algebras. The  study of the inner derivations of nilpotent Leibniz algebras is a very difficult problem. Therefore, we consider some subclasses of these nilpotent algebras. We study almost inner derivations of the null-filiform Leibniz algebras, and also consider almost inner derivations of the some filiform Leibniz algebras.

\section{Preliminaries}

\begin{definition}
An algebra $L$ over a field $F$ is called the \emph {Leibniz algebra} if for all $x,y,z\in L$ the Leibniz identity holds:
$$[x,[y,z]]=[[x,y],z]-[[x,z],y],$$
where $[ \ , \ ]$ is the multiplication in $L.$
\end{definition}

For an arbitrary Leibniz algebra $ L $, we define a sequence:
$$L^1=L, \,\ L^{k+1}=[L^k,L^1], \,\ k\geq 1.$$

The Leibniz algebra $L$ is said to be \emph {nilpotent} if there exists $s\in \mathbb{N}$ such that $L^s=0$. The minimal number $s$ with this property is called the \emph {nilpotency index} or \emph {nilindex} of the algebra $L$.

We recall that the Leibniz algebra is called

\emph{null-filiform}, if $dim L^i=(n+1)-i,$ $1 \leq i \leq n+1;$

\emph{filiform}, if $dim L^i=n-i,$ $2 \leq i \leq n$.

Let $L$ be a nilpotent Leibniz algebra with nilindex $s$.

We consider $L_i = L^i/L^{i+1}, \ 1\leq i \leq s-1$ and $grL = L_1\oplus L_2 \oplus ... \oplus L_{s-1}$. Then $[L_i, L_j] \subseteq L_{i+j}$ and we obtain the graded algebra $grL$.

\begin{definition} If the Leibniz algebra $L$ is isomorphic algebra $grL$, then $L$ is called naturally graded Leibniz algebra.
\end{definition}

For the Leibniz algebra $L$, we denote the \emph{right} and \emph{left} annihilators, respectively, as follows
$$Ann_r(L)=\{x\in L \ | \ [L,x]=0\}, \quad Ann_l(L)=\{x\in L \ | \ [x,L]=0\}.$$
We denote the \emph{center} of the algebra by $Cent(L)=Ann_r(L) \cap Ann_l(L). $

A linear map $d$ is called a \emph{derivation} of the Leibniz algebra $L$, if
$$d([x,y])= [d(x),y] + [x,d(y)].$$

We denote the space of all derivations by $Der(L)$.

For each $x \in L$, the operator $R_x: L \rightarrow L$ which is called the \emph{right multiplication}, such that $R_x(y)=[y,x], \ y\in L,$ is a derivation. This derivation is called \emph{an inner derivation} of $L$, and we denote the space of all inner derivations by $Inner(L)$.

\begin{definition}
The derivation $D \in Der(L)$ of the Leibniz algebra $L$ is called {almost inner derivation}, if $D(x) \in [x,L]$ $([x,L]\subseteq L)$ holds for all $x \in L$; in other words, there exists $a_x \in L$ such that $D(x)=[x,a_x].$
\end{definition}
The space of all almost inner derivations of $L$ is denoted by $AID(L).$

\begin{definition}
The derivation $D \in AID(L)$ of the Leibniz algebra $L$ is called \textbf{the right central almost inner derivation}, if there exists $x \in L$ such that the map $(D-R_x): L \rightarrow Ann_r(L)$.
\end{definition}
The space of right central almost inner derivations of $L$ is denoted by $RCAID(L)$, respectively.

\begin{definition} The derivation $D \in AID(L)$ of the Leibniz algebra $L$ is called \textbf{central almost inner derivation}, if there exists $x\in L$ such that the map $(D-R_x): L \rightarrow Cent(L).$
\end{definition}
The space of central almost inner derivations of $L$ is denoted by $CAID(L),$ respectively.

%\begin{definition}\label{def5} The derivation $D \in AID(L)$ of the Leibniz algebra $L$ is called:
%
%\textbf{right central almost inner derivation}, if there exists $x \in L$ such that the map $(D-R_x): L \rightarrow Ann_l(L)$.
%
%%\textbf{left central almost inner derivation}, if there exists $x \in L$ such that the map $(D-L_x): L \rightarrow Ann_r(L)$;
%%
%%\textbf{left central almost inner derivation}, if there exists $x \in L$ such that the map $(D-L_x): L \rightarrow Ann_l(L)$.
%
%\end{definition}
%
%Since we study the right Leibniz algebras, then we consider the right multiplication and further the definition \ref{def5} is not necessary.

\

\section{Main results}

\subsection{The properties of almost inner derivations of the Leibniz algebras}

The subspaces $Inner(L), \ CAID(L), \ RCAID(L), \ AID(L), \ Der(L)$ are  Lie subalgebras with $[D,D']=DD'-D'D$.
\begin{proposition}
We have the following inclusions of Lie subalgebras
$$Inner(L) \subseteq CAID(L) \subseteq RCAID(L) \subseteq AID(L) \subseteq Der(L).$$
\end{proposition}
\begin{proof}
Let $D_1, D_2 \in AID(L)$ and $x\in L$. Then there exist $y_1,y_2 \in L$ such that $D_1(x)=[x, y_1], \ D_2(x)=[x, y_2]$. Using the property of the derivation and the Leibniz identity, we get the following
$$[D_1,D_2](x)=(D_1D_2)(x)-(D_2D_1)(x)=[D_1(x),y_2]+[x,D_1(y_2)]-[D_2(x),y_1]-[x,D_2(y)]=$$
$$[[x,y_1],y_2]-[[x,y_2],y_1]+[x,D_1(y_2)]-[x,D_2(y_1)]=$$
$$=[x,[y_1,y_2]]+[x,D_1(y_2)]-[x,D_2(y_1)]=[x, [y_1,y_2]+D_1(y_2)-D_2(y_1)].$$

Therefore, $[D_1,D_2](x)=[x,[y_1,y_2]+D_1(y_2)-D_2(y_1)] \in [x,L],$ we have $[D_1,D_2] \in AID(L).$

Let $C_1,C_2 \in CAID(L)$. Then there exist $y_1,y_2 \in L$ such that $C_1-R_{y_1}$ and $C_2-R_{y_2}$ are maps from $L$ to $Cent(L)$. We consider $[C,R_x]=R_{C(x)}$ for $C \in Der(L)$ and obtain the following
$$[C_1-R_{y_1}, C_2-R_{y_2}]=[C_1,C_2]-[C_1,R_{y_2}]-[R_{y_1},C_2]+[R_{y_1},R_{y_2}]=$$
$$=[C_1,C_2]-R_{C_1(y_2)}+R_{C_2(y_1)}-R_{[y_2,y_1]}=[C_1,C_2]-(R_{C_1(y_2)}-R_{C_2(y_1)}+R_{[y_2,y_1]}).$$
Hence we have that the linear transformation $[C_1,C_2]-(R_{C_1(y_2)}-R_{C_2(y_1)}+R_{[y_2,y_1]})$ maps $L$ to $Cent(L)$. Hence $[C_1,C_2] \in CAID(L).$

Let $D_1,D_2 \in RCAID(L)$. Then there exist $y_1,y_2 \in L$ such that $D_1-R_{y_1}$ and $D_2-R_{y_2}$ are maps $L$ to $Ann_r(L)$. We consider $[D,R_x]=R_{D(x)}$ for $D \in Der(L)$ and obtain the following
$$\begin{array}{lll}
[D_1-R_{y_1},
D_2-R_{y_2}]=[D_1,D_2]-[D_1,R_{y_2}]-[R_{y_1},D_2]+[R_{y_1},R_{y_2}]=\\
=[D_1,D_2]-R_{D_1(y_2)}+R_{D_2(y_1)}-R_{[y_2,y_1]}=[D_1,D_2]-(R_{D_1(y_2)}-R_{D_2(y_1)}+R_{[y_2,y_1]}).
\end{array}$$
Hence we have that the linear transformation $[D_1, D_2]-(R_{D_1(y_2)} - R_{D_2(y_1)} + R_{[y_2, y_1]})$ maps $L$ to $Ann_r(L)$. Hence $[D_1, D_2] \in RCAID (L).$
%%%%%%%%%%%%%%%%%%%%%%%%555

Now let us show that $Inner(L) \subseteq CAID(L).$ Let $R_x,R_y \in Inner(L)$ and $R_x-R_y: L \rightarrow Cent(L)$. For every $z \in L, \ a \in Cent(L)$ we consider the following
$$(R_x-R_y)(z)=[z,x]-[z,y]=[z,x]-[z,a+x]=[z,a]\in Cent(L).$$
Therefore, $Inner(L) \subseteq CAID(L).$
%%%%%%%%%%%%%%%%%%%%%%%%%
\end{proof}

\begin{proposition}
The subalgebra $RCAID(L)$ is a Lie ideal in $AID(L)$.
\end{proposition}

\begin{proof}
Let $C \in RCAID(L)$ and $D \in AID(L)$. We must show $[D, C] \in RCAID(L)$. We already know $[D, C] \in AID(L)$. We fix an element $x \in L$ such that $C ': = C-R_x$ maps $L$ to $Ann_r(L)$. We denote $D':=[D, C] -R_{D(x)}$. Then from $[D,R_x]=R_{D(x)}$ we obtain
$$[D,C']=[D,C-R_x]=[D,C]-[D,R_x]=[D,C]-R_{D(x)}=D'$$
and $D'$ maps $L$ to $Ann_r(L)$. Hence for all $y \in L$ we have
$$D'(y)=[D,C'](y)=D(C'(y))-C'(D(y)),$$
because $C'$ maps $L$ to $Ann_r(L)$ and $D$ maps $Ann_r(L)$ to $Ann_r(L)$.
\end{proof}

\begin{proposition}
Let $L$ be the Leibniz algebra. Then the followings are true:

1) Let $D\in AID(L)$. Then $D(L)\subseteq [L, L], \ D(Cent(L))=0$ and $D(I) \subseteq I$ for every ideal $I$ of $L$.

2) For $D \in CAID(L)$, there exists an $x \in L$ such that $D_{|[L, L]} = {R_x} _ {| [L, L]}$.

3) If $L$ has nilindex $3$, then $CAID(L)=AID(L)$.

4) If $Cent(L)=0$, then $CAID(L)=Inner(D)$.

5) If $L$ is nilpotent, then $AID(L)$ is nilpotent.

6) $AID(L \oplus L')=AID(L) \oplus AID(L')$.

\end{proposition}

\begin {proof}
1) By definition, almost inner derivations of $L$ maps to $[L,L]$ and $Cent(L)$ to 0.

Let $x\in I$. Then we have $D(x) \in [x,L] \subseteq [I,L]\subseteq I$.

2) For a given $D\in CAID(L)$, there exists $x \in L$ such that $D'=D-R_x$ satisfies $D'(L) \subseteq Cent(L)$. Hence $D'$ is derivation and for all $u,v \in L$ we have
$$D'([u,v])=[D'(u),v]+[u,D'(v)]=0.$$

3) If $L$ is nilpotent with nilindex 3, i.e. $L^3=0$, then for each $D\in AID(L)$ we get $D(L) \subseteq [L, L] \subseteq Cent(L)$ and get equality.

4) We suppose $Cent(L)=0$ and $D\in CAID(L)$. Then there is $x \in L$ such that $D-R_x=0$. Therefore $D$ is inner.

5) Let $D \in AID(L)$ and $x \in L$. Then $D^k(x) \in [[[..., [x, L], ... L], L]$ ($k$ times $L$). If $k$ is higher than nilpotent
class over $L$, then we have $D^k(x)=0$, therefore $D$ is nilpotent. By Engel's theorem for Leibniz algebras \cite{Ayupov1}, $AID(L)$ is nilpotent.

6) Let $D \in AID(L \oplus L')$. Then the constraints are again almost inner derivations, i.e. $D_{| L} \in AID(L)$ and
$D_ {| L '} \in AID(L')$. It is obvious that the mapping $D \mapsto D_ {| L} \oplus D_ {|L'}$ gives a one-to-one correspondence between
$AID(L \oplus L')$ and $AID(L) \oplus AID(L')$.

\end{proof}

\subsection{Almost inner derivations of null-filiform Leibniz algebras}

Firstly we consider a certain class of nilpotent Leibniz algebras, the so-called null-filiform Leibniz algebra \cite{Ayupov3}.

In any $n$-dimensional null-filiform Leibniz algebra $L$ there exists a basis $\{e_1, e_2, ..., e_n \}$ such that the multiplication in $L$ has the form:
\begin{equation}\label{multip null-fil}
NF_n: \quad [e_i, e_1]=e_{i+1}, \quad 1\leq i\leq n-1
\end{equation}
(the omitted of products are equal to zero).

Let $L$ be a null-filiform Leibniz algebra.
\begin{proposition}
For the $n$-dimensional null-filiform Leibniz algebra $NF_n$ the following equality holds:
$$AID(NF_n)=Inner(NF_n).$$
\end{proposition}

\begin{proof} The null-filiform algebra $L$ is a one-generated algebra, i.e. generated by $e_1$. Let $D\in AID(NF_n)$. Then, by the definition of almost inner derivation, there exists $a_{e_1}$ such that $D(e_1)=R_{a_{e_1}}.$ Let $D'\in AID(NF_n)$ and let $D'=D-R_{a_{e_1}}$, then we get $D'(e_1)=0$.
Then by multiplication (\ref{multip null-fil}) we have
$$D'(e_i)=D'([e_{i-1},e_1])=[D'(e_{i-1},e_1)]+[e_{i-1},D'(e_1)]=0, \ \ 2\leq i \leq n.$$
This means that $$AID(NF_n)=Inner(NF_n).$$

\end{proof}

\subsection{Almost Inner Derivation of Non-Lie Filiform Leibniz Algebras}
%%%%%%%%%%%%%%%%%%%%%%%%%%%%%%%%%%%%%%%%%%%%%%%%%%%%%%%%%%%%%%%%%%%%%%%%%%%%%%%%%55
Now we consider filiform non-Lie Leibniz algebras $F_1(\alpha_4, \alpha_5, ..., \alpha_n, \theta)$ and $F_2(\beta_5, ..., \beta_n, \gamma)$ from \cite{Ayupov3}:

$F_1(\alpha_4, \alpha_5, ..., \alpha_n, \theta) : \  \left\{%
\begin{array}{lll}
[e_1,e_1]=e_3, &  \\ [1mm]

[e_i,e_1]=e_{i+1}, \quad  2\leq i \leq n-1, \\ [1mm]

[e_1,e_2]=\sum\limits_{s=4}^{n-1}\alpha_s e_s + \theta e_n,&  \\ [2mm]

[e_j,e_2]=\sum\limits_{s=4}^{n-j+2}\alpha_s e_{s+j-2},  \ \ 2\leq j \leq n-2, &  \\
\end{array}%
\right.$

\

$F_2(\beta_4, \beta_5, ..., \beta_n, \gamma) : \  \left\{%
\begin{array}{lll}
[e_1,e_1]=e_3, &  \\ [1mm]

[e_i,e_1]=e_{i+1}, \quad  3\leq i \leq n-1, \\ [1mm]

[e_1,e_2]=\sum\limits_{k=4}^{n}\beta_k e_k, &
\\ [1mm]

[e_2,e_2]=\gamma e_n, \\ [1mm]

[e_i,e_2]=\sum\limits_{k=4}^{n+2-i}\beta_k e_{k+i-2}, \ \ 3\leq i \leq n-2. &  \\
\end{array}%
\right.$\\

Let $L$ be an algebra from $F_1(\alpha_4, \alpha_5, ..., \alpha_n, \theta)$ or $F_2(\beta_4, \beta_5, ..., \beta_n, \gamma)$.

Let $L$ be the Leibniz algebra and $E_{n,2}: L \rightarrow L$ be a linear mapping such that
\begin{equation}\label{E(n,2)}
E_{n,2}(e_i)=\delta_{i,2}e_n, \ \ 1\leq i \leq n,
\end{equation}
where $\delta_{i,2}=\left\{
                    \begin{array}{ll}
                      1, & i=2 \\
                      0, & i\neq 2
                    \end{array}
                  \right.
$ -- Kronecker symbol.

\begin{theorem}\label{th F1F2}
Let $L$ be a non-Lie filiform Leibniz algebra and let $D\in AID(L)$. Then there exist an element $x\in L$ and $\lambda\in\mathbb{C}$ such that
$$D-R_x=\lambda E_{n,2}.$$
\end{theorem}

\begin{proof}

We first consider the non-Lie filiform Leibniz algebra $L=F_1(\alpha_4, \alpha_5, ..., \alpha_n, \theta).$

Let $D\in AID(L)$. This algebra is a two-generated algebra, i.e. we have generators $e_1$ and $e_2$. Then, by the definition of almost inner derivation, there exists $a_{e_1}$ such that $D(e_1)=R_{a_{e_1}}$. Let $D'\in AID(L)$ and $D'=D-R_{a_{e_1}}$, then we get $D'(e_1)=0$.
Since $D'(e_1)=0$, then we have the following:
$$\begin{array}{ll}
D'(e_3)=D'([e_1,e_1])=[D'(e_1),e_1]+[e_1, D'(e_1)]=0, \\
D'(e_{i})=D'([e_{i-1},e_1])=[D'(e_{i-1}),e_1]+[e_{i-1},D'(e_1)]=[D'(e_{i-1}),e_1]=0, \ \ 4\leq i \leq n.
\end{array}$$
Let $D'(e_2)=\sum\limits_{j=1}^n b_{j}e_j$. we check the following:
$$D'(e_3)=D'([e_2,e_1])=[D'(e_2),e_1]=[\sum\limits_{j=1}^n b_{j}e_j,e_1]=(b_{1}+b_{2})e_3+b_{3}e_4+...+b_{n-1}e_n.$$
On the other hand, $D'(e_3)=D([e_1,e_1])=0$.
So we get $$b_{1}=-b_{2}, \ \ b_{i}=0, \ 3\leq i\leq n-1.$$
Now we check the following:
$$\begin{array}{lll}
0&=&D'([e_1,e_2])=[D'(e_1),e_2]+[e_1,D'(e_2)]=[e_1, b_{1}e_1-b_{1}e_2+b_{n}e_n]=\\
&=&b_{1}e_3-b_{1}(\alpha_4 e_4 +...+\alpha_{n-1}e_{n-1}+\theta e_n).
\end{array}$$
We have $b_{1}=0$ and $D'(e_2)=b_{n}e_n$. On the other hand, by definition of almost inner derivation
$$b_{n}e_n=D'(e_2)=[e_2, a_{e_2}]=[e_2, a_{2,1}e_1+a_{2,2}e_2+...+a_{2,n}e_n]=a_{2,1}e_3+a_{2,2}(\alpha_4e_4+\alpha_5e_5+...+\alpha_ne_n).$$ We obtain
\begin{equation}\label{D(e2)=a22alphan}
\left\{
\begin{array}{ll}
a_{2,1}=0, \ \ a_{2,2}\alpha_i=0, \ 4\leq i \leq n-1, \\
b_n=a_{2,2}\alpha_n.
\end{array}\right.
\end{equation}
Hence $D'(e_2)=a_{2,2}\alpha_n e_n$. If $a_{2,2}\alpha_n=0$, then $AID(L)=Inner(L)$, so
$$a_{2,2}\alpha_n\neq0,$$
therefore from (\ref{D(e2)=a22alphan}) we get
$$\alpha_i=0, \ 4\leq i \leq n-1.$$
In the end we obtain $D'=a_{2,2}\alpha_n E_{n,2}=\lambda E_{n,2}$.

\

Let $L=F_2(\beta_4, \beta_5, ..., \beta_n, \gamma)$ and $D' \in AID(L)$. By definition AID for $e_2$ there exists $a_{e_2}$ such that
$$D'(e_2)=[e_2,a_{e_2}]=[e_2,a_{2,1}e_1+...+a_{2,n}]=a_{2,2}\gamma e_n.$$

Conducting analogously reasoning in this algebra we obtain $D'(e_1)=0,$ $D'(e_i)=0, \ 3\leq i\leq n$ and $D'=a_{2,2}\gamma E_{n,2}=\lambda E_{n,2}$, where $\lambda\in\mathbb{C}$.

Now we consider the following equality:
$$
\begin{array}{lll}
a_{2,2}\gamma e_n&=&D'(e_1)+D'(e_2)=D'(e_1+e_2)=[e_1+e_2, c_{e_1+e_2}]=[e_1+e_2, c_1e_1+c_2e_2]=\\ &=&c_1e_3+c_2\beta_4e_4+c_2(\beta_4+\beta_5)e_6+...+c_2(\beta_4+...+\beta_{n-1})e_{n-1}+c_2(\beta_4+...+\beta_{n-1}+\beta_n+\gamma)e_n.
\end{array}$$
We get
$$\left\{
\begin{array}{lll}
c_1=0,\\
c_2\beta_i=0, \ 4\leq i \leq n-1, \\
c_2(\beta_n+\gamma)=a_{2,2}\gamma.
\end{array}\right.
$$

If at least one of $\beta_{i_0}\neq0 \ (4\leq i_0\leq n-1)$, then we have $c_2=0$, hence $AID(L)=Inner(L)$. Therefore $\beta_i=0, \ 4\leq i\leq n-1$.

Thus, for filiform non-Lie algebras we obtain $D-R_a=\lambda E_{n,2}, \ \lambda\in\mathbb{C}.$
\end{proof}

\begin{remark}\label{rem1} Let $L$ be a filiform non-Lie Leibniz algebra. If at least one of $\alpha_{i_0}\neq0$ and $\beta_{j_0}\neq 0,$ $i_0, j_0\in\{4,5,...,n-1\}$, then we get $AID(L)=Inner(L).$
\end{remark}
%%%%%%%%%%%%%%%%%%%%
\begin{theorem}\label{AIDF1F2}
Let $L$ be an $n$-dimensional filiform non-Lie Leibniz algebra $F_1(0, ...,0, \alpha_n, \theta)$ or $F_2(0, ...,0, \beta_n, \theta)$. Then at run $\theta=0,\ \alpha_n\neq 0$ and $\beta_n=0, \ \gamma\neq0$ respectively we obtain
$$AID(L)=Inner(L) \oplus \langle E_{n,2} \rangle,$$
where $E_{n,2}$ is the matrix of the elements in which in the place $(n,2)$ we have 1, and other elements are 0.
\end{theorem}

\begin{proof}
Let $L=F_1(0, ...,0, \alpha_n, \theta)$. We have to show that $E_{n,2}$ is an almost inner derivation of the algebra $L$. We take the element $x=\sum \limits_{i=1}^nx_i e_i \in L$, then there is $c_x=c_1e_1+c_2e_2\in L$ and we check up the following
$$\begin{array}{lll}
E_{n,2}(x)&=&[x,c_x]=[\sum\limits_{i=1}^n x_i e_i, c_1 e_1+c_2 e_2]=\\
&=&c_1(x_1+x_2)e_3 + c_1x_3e_4 + c_1x_4e_5+...+c_1x_{n-2}e_{n-1}+(c_1x_{n-1}+c_2(x_1\theta+x_2\alpha_n))e_n.
\end{array}$$

If $\theta\neq0$ and $x_3\neq0$, then for $x_1=-\frac{x_2\alpha_n}{\theta}$ the map $E_{n,2}$ is not almost inner derivation.

Therefore $\theta=0$ and for any $x\in L$ choosing $c_1=0, \ c_2=\frac{1}{\alpha_n}$ we have $E_{n,2}(x)=x_2e_n$. Hence $E_{n,2}\in AID(L)$.
\

Let $L=F_2(0, 0, ...,0, \beta_n, \gamma)$. Let $\forall x=\sum \limits_{i=1}^nx_i e_i \in L$, then $\exists c_x=c_1e_1+c_2e_2\in L$ and we obtain the following
$$\begin{array}{lll}
E_{n,2}(x)&=&[x,c_x]=[\sum\limits_{i=1}^n x_i e_i, c_1 e_1+c_2 e_2]=\\
&=&c_1x_1e_3 + c_1x_3e_4 + c_1x_4e_5+...+c_1x_{n-2}e_{n-1}+(c_1x_{n-1}+c_2(x_1\beta_n+x_2\gamma))e_n.
\end{array}$$

If $\beta_n\neq0$ and $x_4\neq0$, then for $x_1=-\frac{x_2\gamma}{\beta_n}$ the derivation $E_{n,2}$ is not almost inner derivation.

Therefore $\beta_n=0$ and for any $x\in L$ choosing $c_1=0, \ c_2=\frac{1}{\gamma}$ we have $E_{n,2}(x)=x_2e_n$. Hence $E_{n,2}\in AID(L)$.

\end{proof}

Theorem \ref{th F1F2} and \ref{AIDF1F2} imply the following consequence:

\begin{corollary}
In filiform non-Lie Leibniz algebras, if all parameters are equal to zero, then these algebras turn into a graded algebra. Then the almost inner derivations of graded non-Lie Leibniz algebras coincide with the inner derivations.
\end{corollary}

%%%%%%%%%%%%%%%%%%%%%%%%%%%%%%%%% Third class%%%%%%%%%%%%%%%%%%%%%%%%%%%%%%%%%%%%%%%%

\subsection{Almost Inner Derivations of Some Filiform Leibniz Algebras}

We consider filiform Leibniz algebra $L=F_3(\theta_1, \theta_2, \theta_3)$, which contain filiform Lie algebra \cite{Bratzlavsky}:
$$F_3(\theta_1, \theta_2, \theta_3) : \  \left\{%
\begin{array}{llllll}
[e_1,e_1]=\theta_1e_n, \quad [e_1,e_2]=-e_3+\theta_2 e_n, \quad [e_2,e_2]=\theta_3 e_n, \\ [1mm]
[e_i,e_1]=e_{i+1}, \quad  2\leq i \leq n-1, \\ [1mm]
[e_1,e_i]=-e_{i+1}, \quad  3\leq i \leq n-1, \\ [1mm]
[e_i,e_2]=-[e_2,e_i]=\sum\limits_{k=5}^{n-i+3}\beta_k e_{k+i-3}, \ \  3\leq i \leq n-2, \\ [1mm]
[e_i,e_j]=-[e_j,e_i]=0,  \ \ i, j \geq 3. &  \\
\end{array}%
\right.$$

\begin{theorem}\label{th F3}
Let $L=F_3(\theta_1, \theta_2, \theta_3)$ and let $D\in AID(L)$. Then there exist an element $x\in L$ and $\lambda\in\mathbb{C}$ such that
$$D-R_x=\lambda E_{n,2}.$$
\end{theorem}

\begin{proof}

Let $L=F_3(\theta_1, \theta_2, \theta_3)$. Let $D\in AID(L)$. Then $D$ induces an almost inner derivation of $\bar{D}$ by $L / \langle e_n \rangle$. By induction, we can assume that after changing $D$ to inner derivation, we have $\bar{D}=\mu E_{n-1,2}$ for some $\mu \in \mathbb{C}$. This implies such that $D(e_1)=\alpha e_n$ for some $\alpha \in \mathbb{C}$. Now we replace $D$ with $D'=D+R_{\alpha e_{n-1}}$. Then we have
$$\begin{array}{ll}
D'(e_1)=D(e_1)+R_{\alpha e_{n-1}}(e_1)=\alpha e_n + [e_1,\alpha e_{n-1}]=0, \\
D'(e_i)=D(e_i)+[e_i,\alpha e_{n-1}]=D(e_i), \ i\geq 2.
\end{array}$$
We get
$$D'(e_2)=D(e_2)=\mu e_{n-1}+\lambda e_n, \ \ \mu,\lambda\in \mathbb{C}.$$
Hence, we have the following
$$\begin{array}{l}
D'(e_3)=D'([e_2,e_1])=[D'(e_2),e_1]=[\mu e_{n-1}+\lambda e_n,e_1]=\mu e_n,\\
D'(e_4)=D'([e_3,e_1])=[D'(e_3),e_1]=[\mu e_n,e_1]=0,
\end{array}$$
moreover, $D'(e_i)=0, \ i\geq 5$.

Since we have $D'(e_3)=\mu e_n$ and $D'\in AID(L)$, then there exists an element $a_{e_3}=a_{3,1}e_1+a_{3,2}e_2\in L$ such that $D'(e_3)=[e_3,a_{e_3}]=\mu e_n$. Therefore we get the following
$$\mu e_n=[e_3,a_{3,1}e_1+a_{3,2}e_2]=a_{3,1}e_4+a_{3,2}(\beta_5 e_5 + \beta_6 e_6 +...+\beta_n e_n).$$
We obtain
$$a_{3,1}=0, \ \ a_{3,2}\beta_i=0, \ \ 5\leq i \leq n-1, \ \ a_{3,2}\beta_n=\mu. $$
Since we assume $\mu \neq0$, then we have
$$\beta_i=0, \ \ 5\leq i \leq n-1.$$
Now we consider the following
$$D'(e_2)=[e_2,a_{e_2}]=[e_2,\sum\limits_{j=1}^n a_{2,j}e_j]=a_{2,1}e_3+(a_{2,2}\theta_3-a_{2,3}\beta_n)e_n.$$
On the other hand $D'(e_2)=\mu e_{n-1}+\lambda e_n$. We have
$$a_{2,1}e_3+(a_{2,2}\theta_3-a_{2,3}\beta_n)e_n=\mu e_{n-1}+\lambda e_n.$$
Since we assume that $\mu \neq 0$, this equation does not have a solution, which is a contradiction.
Hence indeed $\mu=0$, and therefore $D'= \lambda E_{2,n}$.
\end{proof}

\

\begin{proposition}\label{AIDF3}
Let $L$ be an $n$-dimensional filiform Leibniz algebra $F_3(\theta_1, \theta_2, \theta_3).$ Then
$$AID(F_3(\theta_1, \theta_2, \theta_3))=Inner(F_3(\theta_1, \theta_2, \theta_3))\oplus \langle E_{n,2} \rangle,$$
where $E_{n,2}$ is the matrix of the elements in which in the place $(n,2)$ we have 1, and other elements are 0.
\end{proposition}
\begin{proof}
The proof is analogous to Proposition 7.4 in \cite{Burde}.

\end{proof}

%%%%%%%%%%%%%%%%%%%%% Example %%%%%%%%%%%%%%%%%%%%%%%%%%%%%%%%%%%%%%%%%

\subsection{Almost inner derivations of low dimensional nilpotent Leibniz algebras}

\

N.Jacobson proved the following theorem \cite{Jacobson2}:

\begin{theorem}\label{th-other-der}
Every nilpotent Lie algebra has a derivation $D$ which is not inner.
\end{theorem}

There is a question: Are almost inner derivations of nilpotent Lie algebras outer derivations? And is the converse right? Generally this question is open. We give an example which answers in the positive on this question.

\begin{example}\label{AidOth}
We consider 5-dimensional nilpotent Lie algebra in which there exist almost inner derivations which are not inner \cite{Burde}.
\end{example}

1) $\mathfrak{g}_{5,3}: \ [e_1,e_2]=e_4, \ [e_1,e_4]=e_5, \ [e_2,e_3]=e_5$, the omitted products are equal to zero. Derivations, inner derivations and almost inner derivations of this algebra have the following matrix forms respectively:
$$Der(\mathfrak{g}_{5,3})=
\left(
  \begin{array}{ccccc}
    a_{1,1} & 0       & 0        & 0                & 0\\
    a_{1,2} & a_{2,2} & 0        & 0                & 0\\
    a_{1,3} & a_{2,3} & 2a_{1,1} & 0                & 0\\
    a_{1,4} & a_{2,4} & -a_{2,2} & a_{1,1}+a_{2,2}  & 0\\
    a_{1,5} & a_{2,5} & a_{3,5}  & -a_{1,3}+a_{2,4} & 2a_{1,1}+a_{2,2}
  \end{array}
\right),
$$
$$
Inner(\mathfrak{g}_{5,3})=
\left(
  \begin{array}{ccccc}
    0     & 0      & 0       & 0     & 0 \\
    0     & 0      & 0       & 0     & 0 \\
    0     & 0      & 0       & 0     & 0 \\
    \mu_2 & -\mu_1 & 0       & 0     & 0 \\
    \mu_4 & \mu_3  & -\mu_2  &-\mu_1 & 0
  \end{array}
\right), \
AID(\mathfrak{g}_{5,3})=
\left(
  \begin{array}{ccccc}
    0       & 0       & 0        & 0       & 0\\
    0       & 0       & 0        & 0       & 0\\
    0       & 0       & 0        & 0       & 0\\
    a_{1,4} & a_{2,4} & 0        & 0       & 0\\
    a_{1,5} & a_{2,5} & a_{3,5}  & a_{2,4} & 0
  \end{array}
\right).
$$
If $a_{1,4}=a_{1,5}=a_{2,4}=a_{2,5}=0$, then we obtain the matrix of outer derivation of algebra $\mathfrak{g}_{5,3}$:
$$Outer(\mathfrak{g}_{5,3})=
\left(
  \begin{array}{ccccc}
    a_{1,1} & 0       & 0        & 0                & 0\\
    a_{1,2} & a_{2,2} & 0        & 0                & 0\\
    a_{1,3} & a_{2,3} & 2a_{1,1} & 0                & 0\\
    0       & 0       & -a_{2,2} & a_{1,1}+a_{2,2}  & 0\\
    0       & 0       & a_{3,5}  & -a_{1,3}         & 2a_{1,1}+a_{2,2}
  \end{array}
\right).
$$
Therefore, $AID(\mathfrak{g}_{5,3})\subseteq Outer(\mathfrak{g}_{5,3})$ and any almost inner derivation of the algebra $\mathfrak{g}_{5,3}$ is outer. If in $Outer(\mathfrak{g}_{5,3})$ we have $a_{1,1}=a_{1,2}=a_{1,3}=a_{2,2}=a_{2,3}=0$, then the space of all outer derivations coincides with the space of all almost inner derivations.

\

Now we give examples for low dimensional nilpotent Leibniz algebras.

\begin{example}\label{ex1}
Let $L$ be the three-dimensional nilpotent Leibniz algebra:
$$
\begin{array}{lll}
  L_1(\alpha):& [e_2,e_2]=e_1, \quad [e_3,e_3]=\alpha e_1, \quad [e_2,e_3]=e_1, \quad \alpha\in\mathbb{C},  \\
  L_2:&  [e_2,e_2]=e_1, \quad [e_3,e_2]=e_1, \quad [e_2,e_3]=e_1,  \\
  L_3:&  [e_2,e_2]=e_1, \quad [e_3,e_3]=e_1, \quad [e_3,e_2]=e_1, \quad [e_2,e_3]=e_1,  \\
  L_4:&  [e_3,e_3]=e_1, \\
  L_5:&  [e_2,e_3]=e_1, \quad [e_3,e_3]=e_1, \\
  L_6:&  [e_3,e_3]=e_1, \quad [e_1,e_3]=e_2.
\end{array}
$$
For three-dimensional nilpotent Leibniz algebras $L$, the following equality
$$AID(L)=Inner(L)$$
holds.
\end{example}

\begin{example}\label{ex2}
Let $L$ be four-dimensional nilpotent Leibniz algebra. Then from \cite{Albev} there are 28 algebras and we give only those algebras which will be necessary to us:

$
\begin{array}{lllll}
  L_4:    & [e_1,e_1]=e_3, & [e_1,e_2]=\alpha e_4,                    & [e_2,e_1]=e_3, & [e_2,e_2]=e_4,  \\
          &[e_3,e_1]=e_4, &  \alpha\in\{0,1\};\\
  L_9:    & [e_1,e_1]=e_4, & [e_2,e_1]=e_3,                           & [e_2,e_2]=e_4, & [e_1,e_2]=-e_3+2e_4, \\
          & [e_3,e_1]=e_4, & [e_1,e_3]=-e_4, \\
  L_{10}: & [e_1,e_1]=e_4, & [e_2,e_1]=e_3,                           & [e_2,e_2]=e_4, & [e_3,e_1]=e_4, \\
          &[e_1,e_2]=-e_3, & [e_1,e_3]=-e_4; \\
  L_{11}: & [e_1,e_1]=e_4, & [e_1,e_2]=e_3,                           & [e_2,e_1]=-e_3, & [e_2,e_2]=-2e_3+e_4; \\
  L_{12}: & [e_1,e_1]=e_3, & [e_2,e_1]=e_4,                           & [e_2,e_2]=-e_3;   \\
  L_{13}: & [e_1,e_1]=e_3, & [e_1,e_2]=e_4,                           & [e_2,e_1]=-\alpha e_3, & [e_2,e_2]=-e_4; \\
  L_{20}: & [e_1,e_2]=e_4, & [e_2,e_1]=\frac{1+\alpha}{1-\alpha} e_4, & [e_2,e_2]= e_3, & \alpha\in\mathbb{C}\setminus \{1\}. \\
  \end{array}
$
\end{example}

Let us show the calculation of the dimension of almost inner derivations and the inner derivations of these algebras.

$\bullet$ The algebra $L_4$ is a filiform algebra from the class $F_1(0, ..., 0, \alpha_n, \theta)$. Therefore, by Theorem \ref{AIDF1F2} we have: if $\alpha=0$, then $AID(L_4)=Inner(L_4)$, and if $\alpha=1$, then $AID(L_4)=Inner(L_4) \oplus \langle E_{4,2}\rangle$.

$\bullet$ We consider the algebra $L_9$. Let $D\in AID(L_9)$, then by definition AID for $1 \leq i \leq4$ for each $e_i$ there is $a_{e_i}=\sum_{j=1}^4 a_{i,j}e_j$ and we have the following:
$$\begin{array}{ll}
D(e_1)=[e_1,a_{e_1}]=-a_{1,2}e_3+(a_{1,1}+2a_{1,2}-a_{1,3})e_4, & D(e_2)=[e_2,a_{e_2}]=a_{2,1}e_3+a_{2,2}e_4, \\
D(e_3)=[e_3,a_{e_3}]=a_{3,1}e_4, & D(e_4)=[e_4,a_{e_4}]=0.
\end{array}
$$
Since $D$ is derivation, we check the following:
$$a_{3,1}e_4=D(e_3)=D([e_2,e_1])=[D(e_2),e_1]+[e_2,D(e_1)]=a_{2,1}e_4,$$
from here we get $a_{2,1}=a_{3,1}$. Therefore, the matrix AID of this algebra has the following form:
$$AID(L_9)=
\left(
  \begin{array}{cccc}
    0        & 0       & 0       & 0 \\
    0        & 0       & 0       & 0 \\
    -a_{1,2} & a_{2,1} & 0       & 0 \\
    a_{1,1}+2a_{1,2}-a_{1,3}       & a_{2,2} & a_{2,1} & 0 \\
  \end{array}
\right),
$$
hence $dim AID(L_9)=4.$

Now we calculate the dimension of the space of inner derivations. To do this, we take the element $x=\sum\limits_{i=1}^4x_ie_i$ and consider $R_x(e_i), \ (1\leq i\leq 4):$
$$
\begin{array}{ll}
R_x(e_1)=[e_1,x]=-x_2e_3+(x_1+2x_2-x_3)e_4, & R_x(e_2)=[e_2,x]=x_1e_3+x_2e_4, \\
R_x(e_3)=[e_3,x]=x_1e_4, & R_x(e_4)=[e_4,x]=0.
\end{array}
$$
The matrix of inner derivation of algebra $ L_9 $:
$$Inner(L_9)=
\left(
  \begin{array}{cccc}
    0        & 0       & 0       & 0 \\
    0        & 0       & 0       & 0 \\
    -x_2     & x_1     & 0       & 0 \\
    x_1+2x_2-x_3 & x_2 & x_1     & 0 \\
  \end{array}
\right),
$$
hence $dim Inner(L_9)=3.$

From the matrices $AID(L_9)$ and $Inner(L_9)$ it is clear that $AID(L_9)=Inner(L_9) \oplus \langle E_{4,2} \rangle$. Now let's calculate the dimension of $RCAID(L_9)$, for this we take every element of $x=\sum\limits_{i=4}^4 x_i e_i\in L_9$ and
$$(D-R_x)(x)=\left(
  \begin{array}{cccc}
    0             & 0           & 0           & 0 \\
    0             & 0           & 0           & 0 \\
    -a_{1,2}-x_2  & a_{2,1}-x_1 & 0           & 0 \\
    a'_{1,3}-x'_3 & a_{2,2}-x_2 & a_{2,1}-x_1 & 0 \\
  \end{array}
\right)
\left(
  \begin{array}{c}
    x_1 \\
    x_2 \\
    x_3 \\
    x_4 \\
  \end{array}
\right)=\bar{0}.
$$
Then we have $a_{1,2}=x_2, \ a'_{1,3}=x'_3, \ a_{2,1}=x_1, \ a_{2,2}=x_2$. Hence, $dim RCAID(L_9)=3$.

$\bullet$ For algebras $L_{10}, \ L_{11}, \ L_{12}, \ L_{20}$ similarly conducted reasoning and calculated dimension $AID(L)$ and $Inner(L)$.

$\bullet$ Now we consider $L_{13}$ and get the following matrices:

$$
AID(L_{13})=
\left(
  \begin{array}{cccc}
    0        & 0        & 0  & 0 \\
    0        & 0        & 0  & 0 \\
    a_{1,1}  & -a_{2,1} & 0  & 0 \\
    a_{1,2}  & -a_{2,2}  & 0   & 0 \\
  \end{array}
\right), \quad
Inner(L_{13})=
\left(
  \begin{array}{cccc}
    0    & 0   & 0  & 0 \\
    0    & 0   & 0  & 0 \\
    x_1  & -x_1 & 0  & 0 \\
    x_2  & -x_2 & 0  & 0 \\
  \end{array}
\right).
$$
This shows that $dim AID(L_{13})=4, \ \ dimRCAID(L_{13}) =dim Inner(L_{13})=2$, hence we obtain $AID(L_{13})=Inner(L_{13}) \oplus \langle E_{3,2}+E_{4,2} \rangle$.

For other algebras, except those shown, almost inner derivations coincide with inner derivations.

Therefore, we have the following table:

\begin{center}
\begin{tabular}{c|c|c|c|c|c|c}
Algebra  &  dim Inner(L) & dimRCAID(L) & dimAID(L) & dimDer(L) & D \\
\hline \hline
$L_4$   & 2 & 2 & 3 & 4 & $E_{4,2}$ \\
\hline
$L_9$   & 3 & 3 & 4 & 4 & $E_{4,2}$ \\
\hline
$L_{10}$  & 3 &  3 &  4 & 4 & $E_{4,2}$ \\
\hline
$L_{11}$  & 2 & 2 & 3 & 5 & $E_{4,2}$ \\
\hline
$L_{12}$  & 2 & 2 & 3 & 5 & $E_{4,2}$ \\
\hline
$L_{13}$  & 2 & 2 & 4 & 5 & $E_{4,2}+E_{3,2}$ \\
\hline
$L_{20}$  & 2 & 2 & 3 & 7 & $E_{4,2}$ \\
\hline

\end{tabular}
\end{center}

\begin{example}\label{ex3}
Let $L$ be a complex Leibniz algebra of dimension $n \leq 2$. Then we have
$$AID(L)=RCAID(L)=Inner(L).$$

\end{example}

It is clear that for abelian Leibniz algebras $Inner(L)=RCAID(L)=AID(L)=0.$


\begin{thebibliography}{99}


\bibitem{Albev}{Albeverio S., Omirov B.A., Rakhimov I.S.}, {Classification of 4-dimensional nilpotent complex Leibniz algebras}, {\it Extracta Mathematicae}, 2006, 21(3), 197-210.

\bibitem{Ayupov0} {\rm Ayupov Sh.A., Kudaybergenov K.K.}, {Local derivation on finite-dimensional Lie algebras}, {\it Linear Algebra Appl.} 493 (2016) 381-398.

\bibitem{Ayupov01} {\rm Ayupov Sh.A., Kudaybergenov K.K.}, {Local automorphisms on finite-dimensional Lie algebras and Leibniz algebras}, {\it Algebra, Complex Analysis and Pluripotential Theory. USUZCAMP 2017}, Springer Proceedings in Mathematics and Statistics, vol 264. Springer, 2018, 31-44.

\bibitem{AK15}  Ayupov Sh.~A.~, Kudaybergenov K.~K., Rakhimov I.~S., 2-Local derivations on finite-dimensional Lie algebras, \textit{Linear Algebra and its Applications,} 474 (2015), 1-11.

\bibitem{Ayupov1} {\rm Ayupov Sh.A., Omirov B.A.}, {On Leibniz algebras}, {\it Algebra and operator theory (Tashkent, 1997), Kluwer Acad. Publ., Dordrecht}, 1998, pp. 1-12.

\bibitem{Ayupov2} {\rm Ayupov Sh.A., Omirov B.A.}, {On 3-dimensional Leibniz algebras}, {\it Uzbek Math. Journal}, 1999, 9-14

\bibitem{Ayupov3} {\rm Ayupov Sh.A., Omirov B.A.}, {On some classes of nilpotent Leibniz algebras}, {\it Sibirsk. Mat. Zh.} (in Russian), 2001, 42 (1), 18-29 (English translation in Siberian Math. J., 2001, 42(1), 15-24.

\bibitem{Bresar} Bre\'sar M. and \'Semrl P., Mapping which preserve idempotents, local automorphisms, and local derivations, \textit{Canad. J. Math.} 45 (1993) 483-496.

\bibitem{Burde}{\rm Burde D., Dekimpe K., Verbeke B.}, {Almost inner derivation of Lie algebras}, {\it Journal of Algebra and Its Applications} (2018)
1850214 (26 pages) doi: 10.1142/S0219498818502146.

\bibitem{Bratzlavsky} {\rm Bratzlavsky F.}, {\it Classification des alg\'ebres de Lie de dimension n, de classe $n-1$,
dont l'id\'eal d\'eriv\'e est commutativ}, {\it Bull.\ Cl.\ Sci.\ Bruxelles} 60 (1974), 858--865.

\bibitem{Jacobson} Jacobson N., Lie algebras, {\it Interscience Publishers, Wiley, New York}, (1962).

\bibitem{Jacobson2} Jacobson N., A Note on Automorphisms and Derivations of Lie Algebras. {\it In: Nathan Jacobson Collected Mathematical Papers. Contemporary Mathematicians. Birkh\"{a}user Boston} (1989).

\bibitem{Gordon} {Gordon C.S., Wilson E.N.}, {Isospectral defermations of compact solvmanifolds}, {\it J.Differential Geom.}, 19(1) (1984) 214-256.

\bibitem{Hump} {\rm Humphreys J.E.}, {Introduction to Lie algebras and Representation theory}, {\it Springer-Verlag New York}, 1972.

\bibitem{John01} Johnson B.E., Local derivations on C$^*$-algebras are derivations, \emph{Trans. Amer. Math. Soc.} {353}, 313-325 (2001).

\bibitem{Kad90} Kadison R.V., Local derivations, \emph{J. Algebra} {130}, 494-509 (1990).

\bibitem{Khudoyberdiev} {\rm Khudoyberdiev A.K., Ladra M., Omirov B.A.}, {The classification of non-characteristically nilpotent filiform Leibniz algebras}, {\it Algebras and Representation Theory}, (2014) 17(3) 945-969.

\bibitem{LarSou} Larson D.R. and Sourour A.R., Local derivations and local automorphisms of $B(X)$, \emph{Proc. Sympos. Pure Math.} {51}, Part 2, Providence, Rhode Island 1990, pp. 187-194.



\end{thebibliography}
\end{document}